\documentclass[12pt]{amsart}
\usepackage{amsfonts,amsmath,enumerate}
\usepackage{amssymb,amscd,amsthm}
\usepackage[all]{xy}
\usepackage[pdftex]{graphicx}
\theoremstyle{plain}
\newtheorem{theorem}[subsection]{Theorem}

\newtheorem{proposition}[subsection]{Proposition}
\newtheorem{corollary}[subsection]{Corollary}

\theoremstyle{definition}
\newtheorem{remark}[subsection]{Remark}
\newtheorem{definition}[subsection]{Definition}
\newtheorem{example}[subsection]{Example}

\newcommand{\A}{{\mathcal A}}
\newcommand{\C}{\mathbb{C}}
\newcommand{\Q}{\mathbb{Q}}

\numberwithin{equation}{section}

\begin{document}

\title [On Plane Curves with Double and Triple Points]
{On Plane Curves with Double and Triple Points}

\author[Nancy Abdallah]{Nancy Abdallah}
\address{Univ. Nice Sophia Antipolis, CNRS,  LJAD, UMR 7351, 06100 Nice, France.}
\email{Nancy.ABDALLAH@unice.fr}

\subjclass[2000]{Primary  32S35, 32S22; Secondary  14H50}

\keywords{plane curves, ordinary double and triple points, Milnor algebra, syzygies, Hodge and pole order filtrations }

\begin{abstract} 
We describe in simple geometric terms the Hodge filtration on the cohomology $H^*(U)$ of the complement $U=\mathbb{P}^2\setminus C$ of a plane curve $C$ with ordinary double and triple points. Relations to Milnor algebra, syzygies of the Jacobian ideal and pole order filtration on $H^2(U)$ are given.

\end{abstract}

\maketitle

\section{Introduction}

Let $S=\oplus_rS_r=\C[x,y,z]$ be the graded ring of polynomials with complex coefficients, where $S_r$ is the vector space of homogeneous polynomials of $S$ of degree $r$. For a homogeneous polynomial $f$ of degree $N$, define the Jacobian ideal of $f$ to be the ideal $J_f$ generated in $S$ by the partial derivatives $f_x,f_y,f_z$ of $f$ with respect to $x$, $y$ and $z$. The graded \textit{Milnor algebra} of $f$ is given by $$M(f)=\oplus_rM(f)_r=S/J_f.$$
The study of such Milnor algebras is related to the singularities of the projective curve $C\subset \mathbb{P}^2$ defined by $f=0$, see \cite{DCh}, as well as to the mixed Hodge theory of the curve $C$ and that of its complement $U=\mathbb{P}^2\setminus C$, see the foundational article by Griffiths \cite{Gr} and also \cite{Dimca2}, \cite{DS}, \cite{DSW} treating singular hypersurfaces in $\mathbb{P}^n$. For other relations with Algebraic Geometry see \cite{Se}. A key question is to relate the Hodge filtration $F$ to the pole order filtration $P$ (whose definition will be recalled below) on the cohomology group $H^2(U)$.

The Milnor algebra can be seen (up to a twist in grading) as the top cohomology $H^{3}(K^*(f))$ of the Koszul complex of $f_x,f_y$ and $f_z$ defined in section 3.

The aim of this paper is to generalize the results given by A. Dimca and G. Sticlaru in \cite{DSt} on nodal curves to curves whose singularities are nodes (alias ordinary double points $A_1$) and ordinary triple points $D_4$. In the case of nodal curves, by the work of Deligne \cite{Del}, one has the equality $F=P$ between the two filtrations on $H^2(U)$. In the case at hand, this equality no longer holds, and this explains why the results are more involved to state and to prove.

In section 2, we recall some basic facts on the mixed Hodge theory and  extend one of the main results of A. Dimca and G. Sticlaru in \cite{DSt} from the nodal to the double and triple points, see Theorem \ref{result1}. 

In section 3,  we give a brief introduction about syzygies of the Jacobian ideal of $f$ and the relation between Koszul complex cohomology and Hodge theory.

In the final section we state and prove the main result, which is an estimation of the dimensions of certain homogeneous components of the cohomology groups $H^{m}(K^*(f))$ in terms of simple geometrical invariants, like the number of double and triple points and the genera of the irreducible components of $C$, see Theorem \ref{result2}. 

As a consequence, we find examples of curves with only ordinary double and triple points for which $P^2H^2(U)\neq F^2H^2(U)$. The previous known examples involved curves with non-ordinary multiple points, see Examples 3.2, 3.3 and 3.4 in \cite{DSt}. We also discuss in Example \ref{Pappus} the two distinct realization of the Pappus configuration $9_3$ and point out another subtle difference between nodal curves and curves with  ordinary double and triple points.

\section{Hodge Theory of plane curve complements}

Let $X$ be an algebraic variety. Define $H^*_c(X,\C)$ to be the cohomology groups with compact support of $X$. If $X$ is smooth, then $H^m_c(X,\C)$ is dual to $H^{2n-m}(X,\C)$ where $n=\dim_{\C}X$, (see \cite[p.134]{Voisin}). Recall that the mixed Hodge numbers $h^{p,q}(H^s(X))$ are defined by 
$$h^{p,q}(H^s(X))=\dim Gr_F^pGr_{p+q}^WH^s(X,\mathbb{C}),$$ and we have 
\begin{equation}\label{hodge nb duality}
h^{p,q}(H^m(X,\C))=h^{n-q,n-p}(H^{2n-m}_c(X,\C))
\end{equation}
for every $p,q\leq n$.
Consider now  $U=\mathbb{P}^2\setminus C$, a smooth affine variety. By Deligne \cite{Del}, $H^*(U,\C)$ has a mixed Hodge structure. In particular, for the Hodge filtration $F$ on $H^2(U)$
one has
\begin{equation}\label{hodgefilt}
H^2(U)=F^0=F^1 \supset F^2 \supset F^3=0,
\end{equation}
where the second equality is proven in \cite[p.185]{Dbk}.\\

The dimensions of the associated graded groups $Gr_F^mH^2(U)$ with respect to the Hodge filtration are described in the following theorem, which is our first main result.

\begin{theorem}\label{result1}
Let $C\subset \mathbb{P}^2$ be a curve of degree $N$. Suppose that $C$ has  $n$ nodes and $t$ ordinary triple points as singularities and set $U=\mathbb{P}^2 \backslash C$. Let $C=\bigcup_{j=1,r}C_j$ be the decomposition of $C$ as a union of irreducible components, let $\nu_j: \tilde{C}_j \rightarrow C_j$ be the normalization mappings and let $g_j=g(\tilde{C}_j)$ be the corresponding genera. Then one has
$$\dim Gr^1_F H^2(U,\C)=\sum_{j=1}^r g_j$$
and $$\dim Gr^2_F H^2(U,\C)=\frac{(N-1)(N-2)}{2}-t.$$
\end{theorem}

\begin{proof}
Suppose that the curve $C_j: f_j=0$ has degree $N_j$, and has $n_j$ nodes and $t_j$ triple points. Recall the definition of the Hodge-Deligne polynomial of a quasi-projective complex variety $X$
$$P(X)(u,v)=\sum_{p,q} E^{p,q}(X)u^pv^q$$
where $E^{p,q}(X)=\sum_s (-1)^s\dim Gr_F^pGr_{p+q}^WH^s_c(X,\mathbb{C})$, and the fact that it is additive with respect to constructible partitions, i.e. $P(X)=P(X\setminus Y)+P(Y)$ for a closed subvariety $Y$ of $X$.\\
Using the normalization maps $\nu_j$, we have
\begin{eqnarray*}
P(C_j)
&=&P(C_j\backslash(C_j)_{sing})+P((C_j)_{sing})=P(\tilde{C_j}\backslash\{(2n_j+3t_j)\ points\})+n_j+t_j\\
&=&P(\tilde{C_j})-P(\{2n_j+3t_j\ points\})+n_j+t_j\\
&=& uv-g_ju-g_jv+1-n_j-2t_j.
\end{eqnarray*}
Indeed, it is known that for a smooth curve $C$, the genus $g(C)$ is exactly the Hodge number $h^{1,0}(X)=h^{0,1}(X)$. Now,
\begin{eqnarray*}
P(C)&=&P(C_1\cup\cdots\cup C_r)\\
&=&\sum_{j=1}^r P(C_j)-\sum_{1\leq i<j \leq r}P(C_i\cap C_j)+\sum_{1\leq i<j<k \leq r}P(C_i\cap C_j \cap C_k)\\
&=&ruv - \left(\sum_{j=1}^r g_j\right)u - \left(\sum_{j=1}^r g_j\right)v+r - \sum_{j=1}^r (n_j+2t_j)- (\sum_{1\leq i<j \leq r}N_iN_j-s)+t'.
\end{eqnarray*}
where $s$ (respectively $t'$) denotes the total number of triple points which are intersection of only two (respectively three) curves. The term $\sum_{1\leq i<j \leq r}N_iN_j-s$ is a result of B\'ezout's Theorem. Indeed, the total number of intersection points counted with multiplicities is $\sum_{1\leq i<j \leq r}N_iN_j$, and it is easy to see that the intersection multiplicity of each of the $s$ triple points is $2$, i.e., each one is counted twice.
Next we have by the additivity, $P(U)=P(\mathbb{P}^2)-P(C)$, where $P(\mathbb{P}^2)=u^2v^2+uv+1$. Let's look at the cohomology of the smooth surface $U$. The group $H^4_c(U,\C)$ is dual to the group $H^0(U,\C)$, which is 1-dimensional of type (0,0). It follows that $H^4_c(U,\C)$ is 1-dimensional of type (2,2) and its contribution to $P(U)$ is exactly the term $u^2v^2$.\newline
The group $H^3_c(U,\C)$ is dual to the group $H^1(U,\C)$, which is $(r-1)$-dimensional of type (1,1). Indeed, by Theorem (C.24) in \cite{Dbk}, the only nonzero weights on $H^1(U,\C)$ are $m=1$ and $m=2$. But $W_1$ is zero since $W_1=j^*(H^1(\mathbb{P}^2))=0$ where $j:U\longrightarrow \mathbb{P}^2$ is the compactification of $U$ and we apply again Theorem (C.24) in \cite{Dbk}. 
It follows that the contribution of
$H^3_c(U,\C)$ to $P(U)$ is exactly the term $-(r-1)uv$.\newline
The remaining terms come from the group $H^2_c(U,\C)$, which is dual to the group $H^2(U,\C)$. By theorem (C.24) in \cite{Dbk}, the only nonzero weights on $H^2(U,\Q)$ are $m=2,\ 3$ and  $4$. On the other hand, $W_2=0$, since $W_2=j^*H^2(\mathbb{P}^2)$. But $H^2(\mathbb{P}^2)=\alpha\Q$, where $\alpha=c_1(L)$, $c_1$ denotes the Chern class of $L=O(1)$. Therefore, by the naturality property of Chern classes (see \cite{Bott & Tu}), $N\alpha=j^*(c_1(L^{\otimes N}))=c_1(j^*(L^{\otimes N}))=c_1(L^{\otimes N}\vert U)$ which is equal to zero since it has a nowhere vanishing section given by $f$, the defining equation of $C$. This implies that $H^2(U,\C)$ has only classes of type $(2,1)$, $(1,2)$, and $(2,2)$. Therefore the dimension $\dim Gr^1_FH^2(U,\C)$ is the number of independent classes of type (1,2), which correspond to classes of type $(1,0)$ in $H^2_c(U)$, according to equation \ref{hodge nb duality}, 
and hence to the terms in $u$ in $P(U)$. This gives the first equality.\newline
Now the dimension $\dim Gr^2_FH^2(U,\C)$ is the number of independent classes of type (2,1) or (2,2), which correspond respectively to the terms in $v$ or the constant terms in the polynomial $P(U)$. This yields
$$\dim Gr^2_FH^2(U,\C)=\sum_{j=1}^r{(g_j+n_j+3t_j-1)}+\sum_{1\leq i<j \leq r}N_iN_j +1 - (\sum_{j=1}^r t_j+s+t').$$
Recall the formula $$g_j+n_j+3t_j=p_a(C_j)=\frac{(N_j-1)(N_j-2)}{2},$$
where $p_a$ denotes the arithmetic genus, see \cite[ p.298 and p.54]{Hartshorne}. Knowing that the total number of triple points $t=\sum_{j=1}^r t_j+s+t'$, and using the fact that $N=\sum_{j=1}^r N_j$, with some computations, we get:

\begin{eqnarray*}
\dim Gr^2_FH^2(U,\C)&=&\sum_{j=1}^r{\left(\frac{(N_j-1)(N_j-2)}{2}-1\right)}+\sum_{1\leq i<j \leq r}N_iN_j +1 -t\\
%&=&\sum_{j=1}^r{\frac{d_j^2-3d_j}{2}}+\sum_{1\leq i<j \leq r}d_id_j +1 -t\\
%&=&(\sum_{j=1}^r{\frac{d_j^2}{2}}+\sum_{1\leq i<j \leq r}d_id_j)-\sum_{j=1}^r{\frac{3d_j}{2}}+1-t\\
%&=&\frac{1}{2}(\sum_{j=1}^r{d_j})^2-\frac{3d}{2}+1-t=\frac{d^2-3d+2}{2}-t\\
&=&\frac{(N-1)(N-2)}{2}-t.
\end{eqnarray*}
\end{proof}

\begin{remark}\label{rk1}
In the case of nodal cuvres, i.e. for $t=0$, the above theorem was already proved by A. Dimca and G. Sticlaru in \cite{DSt2}.
\end{remark}
\begin{remark}\label{rk1} As shown in the proof above, we have $W_2H^2(U,\C)=0$, and this implies
$$h^{2,0}(H^2(U))=h^{1,1}(H^2(U))=h^{0,2}(H^2(U))=0.$$

\end{remark}
This last remark implies the following consequence.

\begin{corollary}\label{MHnumber} 
With the above notation and assumptions, we have the following.

\begin{enumerate}[(i)]

\item $h^{2,1}(H^2(U))=h^{1,2}(H^2(U))=\sum_{j=1}^r g_j.$

\item  $h^{2,2}(H^2(U))=\frac{(N-1)(N-2)}{2} -\sum_{j=1}^r g_j  -t.$

\item  $b_2(U)=\frac{(N-1)(N-2)}{2}+\sum_{j=1}^r g_j-t,$
where $b_2(U)$ denotes the second Betti number of the complement $U$.

\end{enumerate}
\end{corollary}
In particular, it follows that $H^2(U)$ is pure of type $(2,2)$ when $g_j=0$ for all $j$, a well known property in the case of line arrangements.

\begin{example}
Let $C:\ (x^2-y^2)(y^2-z^2)(x^2-z^2)=0$. $C$ is the union of $6$ lines in $\mathbb{P}^2$. It has $4$ triple points. We have $g_i=0$ for $i=1,\cdots 6$, $N=6$, and $t=4$. Then according to the aforementioned theorem (\ref{result1}) we get $\dim Gr^1_F H^2(U,\C)=\dim \frac{F^1}{F^2}=0$ and $\dim Gr^2_F H^2(U,\C)=\dim F^2=6$. Hence, $b_2(U)=6$.
\end{example}
\begin{example} Let $C:\ (x^3-y^3)(y^3-z^3)(x^3-z^3)=0$. $C$ is the union of $9$ lines. It has $12$ triple points and no nodes. We have $g_i=0$ for every $i=1,\cdots 6$, $N=9$, and $t=12$. Then $\dim Gr^1_F H^2(U,\C)=0$ and $\dim Gr^2_F H^2(U,\C)=16=b_2(U)$.
\end{example}
\begin{example}
Let $C:\ xyz(x^2y+x^2z+y^2x+y^2z+z^2x+z^2y)=0$. $C$ is the union of $3$ lines giving rise to a triangle and a smooth cubic curve. It has $3$ triple points (the vertices of the triangle) and $3$ nodes. We have $g_1=g_2=g_3=0$, $g_4=1$, $N=6$, and $t=3$. Then $\dim Gr^1_F H^2(U,\C)=1$, $\dim Gr^2_F H^2(U,\C)=7$, and $b_2(U)=1+7=8$.
\end{example}

\section{Koszul complexes, syzygies and spectral sequences for reduced plane curves}

Let $K^*(f)$ be the Koszul complex of the partial derivatives $f_x,f_y,f_z$ of $f$ with the natural grading $|x|=|dx|=1$ defined by
$$0\rightarrow\Omega^0\xrightarrow{df\wedge} \Omega ^1\xrightarrow{df\wedge}\Omega ^2\xrightarrow{df\wedge}\Omega ^3\rightarrow 0$$
where $df=f_xdx+f_ydy+f_zdz$, and $\Omega^k$ denotes the global polynomial $k$-forms on $\C^3$.\\
It is easy to see that $H^3(K^*(f))_k=M(f)_{k-3}$. On the other hand, the homogeneous components of $H^2(K^*(f))_{2+m}$ are the syzygies $$R_m:\;\; af_x+bf_y+cf_z=0,$$ where $a,b,c\in S_m$, modulo the trivial syzygies generated by $(f_j)f_i+(-f_i)f_j=0$, where $f_i$ and $f_j$ denote $f_x,f_y$ or $f_z$. Denote by $ER(f)$ the set of these relations, called essential relations, or nontrivial relations.

By Theorem 3.1 in \cite{Dimca}, we have
\begin{equation}\label{ER}
\dim ER(f)_{k-2}=\dim H^2(K^*(f))_{k}
\end{equation} 
for any $2\leq k\leq 2N-3$, and $\dim ER(f)_{k-2}=\tau(C)$ for $k\geq 2N-4$. Here $\tau(C)$ denotes the sum of the Tjurina numbers of the singularities of $C:f=0$. For instance, if $C$ has $n$ $A_1$-singularities and $t$ $D_4$-singularities, then $\tau(C)=n+4t$. We also have, 
\begin{equation}
\dim H^2(K^*(f))_{2N-3-k}=\dim M(f)_{3N-6-k}-\dim M(f_s)_k
\end{equation}

 Recall the following integers introduced (for hypersurfaces in $\mathbb{P}^n$) in \cite{DSt}:
\begin{definition} For a plane curve $C:f=0$ of degree $N$ with isolated singularities we set
\begin{enumerate}[(i)]
\item the \textit{coincidence threshold ct(C)} defined as $$ct(C)=max\{q:\ \dim M(f)_k=\dim M(f_s)_k\ for\ all\ k\leq q\},$$ with $f_s$ a homogeneous polynomial in $S$ of degree $N$ such that $C_s$: $f_s=0$ is a smooth curve in $\mathbb{P}^2$.
\item the \textit{stability threshold st(C)} defined as $$ st(C)=min\{q: \dim M(f)_k=\tau(C)\ for\ all\ k\geq q\}.$$
 \item the \textit{minimal degree of nontrivial syzygy mdr(C)} defined as $$mdr(C)=min\{q:H^2(K^*(f))_{q+2}\neq0\}.$$
\end{enumerate}
It is easy to see that $ct(C)=mdr(C)+N-2$. By Proposition 1 in \cite{DCh} we have $N-2\leq ct(C)\leq 3(N-2)$ and by Theorem 3 in \cite{DCh} we have $st(C)\leq 3N-5$.
\end{definition}

\begin{example}
Let $C:$ $f=x^py^q+z^N$ with $p>0$, $q>0$ and $p+q=N>2$. It is easy to see that $qxf_x-pyf_y=0$. Therefore, the first nontrivial syzygy is of degree one. Hence $\dim ER(f)_1=\dim H^2(K^*(f))_{3}\neq 0$. Then $mdr(C)=1$ which yields  $ct(C)=N-1$. 
\end{example}

Suppose now that $C$ has only nodes as singularities, i.e. singularities of type $A_1$, and let $C=\bigcup_{j=1,r}C_j$ be the decomposition of $C$ as a union of irreducible components. In this case we have $ct(C)\geq 2N-4$, see Theorem 1.2 in \cite{DSt}. Therefore, the dimensions of $M(f)_q$ are all determined for $q<2N-3$. The next dimension is given by
\begin{equation}\label{eqt}
\dim M(f)_{2N-3}=n(C)+\sum_{j=1}^r g_j=g+r-1
\end{equation}
where $n(C)$ is the number of nodes of $C$ and $$g=\frac{(N-1)(N-2)}{2}.$$ If $C$ is a rational curve, i.e. $g_i=0$ for $i=0,\cdots,r$, then $\dim M(f)_{2N-3}=n(C)=\tau(C)$ and therefore $st(C)\leq 2N-3$. We recall the following corollary in \cite{DSt}.
\begin{corollary}\label{rational nodal curves} For a rational nodal curve $C$, the Hilbert-Poincar\'e series 
$$HP(M(f))(t)= \sum_r \dim M(f)_rt^r$$
 is completely determined by the degree $N$ and the number of nodes $n(C)$. In particular, $st(C)=2N-3$ unless $C$ is a generic line arrangement and then $st(C)=2N-4$.
\end{corollary}

\begin{example}
Let $C$ be the degree 4 curve defined by $f=x(x^3+y^3+z^3)$. Then $C$ has 3 collinear nodes. $st(C)\leq 3N-5=7$ and $ct(C)\geq 2N-4=4$. Indeed a computation using Singular \cite{Singular} yields the following Hilbert-Poincar\'e series
$$HP(M(f))(t)=1+3t+6t^2+7t^3+6t^4+4t^5+3(t^6+t^7+\cdots)$$
and hence $ct(C)=4$ and $st(C)=6$.
\end{example}

\begin{example}
Let $C$ be a generic line arrangement defined by $f=xyz(x+y+z)=0$. Then $C$ has $6$ nodes, by Corollory \ref{rational nodal curves}, $HP(M(f))$ is all determined and we have $st(C)=2N-4=4$ and $ct(C)\geq 4$. Therefore $$HP(M(f))(t)=1+3t+6t^2+7t^3+6(t^4+t^5+\cdots),$$
which implies $ct(C)=4$.
\end{example}

Consider now the double complex $(B,d',d'')$ defined by 
$$B^{s,t}=\Omega_{(t+1)N}^{s+t}\;\;\;\; s,t\in \mathbb{Z},$$
$d'=d$, and $d''(\omega)=-|\omega|N^{-1}df\wedge \omega$ for a homogeneous differential form $\omega$. 
Let $(B^*,D_f)$ be the associated total complex, namely, $B^k=\bigoplus_{s+t=k}B^{s,t}$, and $D_f=d'+d''$ with $d'd''+d''d'=0$. Define a decreasing filtration on $B^*$ by $F^pB^k=\bigoplus_{\substack{s\geq p\\ s+t=k}} B^{s,t}$.
With this notation, we have the following result, see \cite[Chapter 6]{Dbk}
\begin{proposition}\label{specsq}
There exists an $E_1$-spectral sequence $(E_r,d_r)$ converging to $H^{p+q-1}(U)$ such that $$E_1^{p,q}(f)=H^{p+q}(K^*(f))_{(q+1)N}.$$
Moreover, the filtration induced by this spectral sequence on $H^*(U)$ coincides with the pole order filtration $P$.
\end{proposition}
In the case of a curve $C\subset \mathbb{P}^2$ with isolated singularities, the only nontrivial cohomology groups of the Koszul complex are $H^2(K^*(f))$ and $H^3(K^*(f))$. Therefore, the nonzero terms of the $E_1$-spectral sequence belong to the lines $p+q=2$ and $p+q=3$. For the terms on the line $p+q=3$, we have 
$$\dim E_1^{p,q}(f)=\dim H^3(K^*(f))_{(q+1)N}=\dim M(f)_{(q+1)N-3}.$$
 For the terms of the line $p+q=2$, 
$$\dim E_1^{p,q}(f)=\dim H^2(K^*(f))_{(q+1)N}=\dim M(f)_{(q+2)N-3}-\dim M(f_s)_{(q+2)N-3}$$ (see \cite{Dimca}).

\section{Curves with $A_1$ and $D_4$ singularities}

We can now state our second main result.

\begin{theorem}\label{result2}
Let $C\in \mathbb{P}^2$ be a curve of degree $N$. Suppose $C$ has $n$ nodes $(A_1)$ and $t$ triple points $(D_4)$ and no other singularities. Let $C=\bigcup_{j=1,r}C_j$ be the decomposition of $C$ as a union of irreducible components, let $\nu_j: \tilde{C}_j \rightarrow C_j$ be the normalization mappings and set $g_j=g(\tilde{C}_j)$. Then we have the following.
\begin{enumerate}[(A)] 
\item $0 \leq \dim M(f)_{2N-3}-\tau(C) \leq \sum_{j=1}^rg_j,$ where $\tau(C)$ is the sum of all Tjurina numbers of the singularities of $C$.
Moreover, the equality $\dim M(f)_{2N-3}-\tau(C)=\sum_{j=1}^rg_j$ holds if and only if the Hodge filtration $F$ and the pole order filtration $P$ on $H^2(U)$ satisfy $F^2H^2(U)=P^2H^2(U)$.
 In particular, 
if all $g_i=0$, one has $\dim M(f)_{2N-3}=\tau(C)$, i.e. $st(C)\leq 2N-3$ and $F^2H^2(U)=P^2H^2(U)$.
\item $max(r-1+t-\sum_{j=1}^r g_j,r-1)\leq \dim ER(f)_{N-2} \leq r-1+t$. In particular, $\dim ER(f)_{N-2}=r-1+t$ if $g_j=0$ for all $j$.
\end{enumerate}
\end{theorem}

\begin{proof}
$(A)$ Consider the spectral sequence of Proposition \ref{specsq} 
$$E_1^{p,q}(f)=H^{p+q}(K^*(f))_{(q+1)N}$$
that converges to $H^{p+q-1}(U)$.
%The only nonzero terms of the $E_1$-spectral sequence are on the lines $L:\ p+q=2$ and $L':\ p+q=3$.
By Theorem 2.4 (ii) in \cite{DSt}, the differential 
$$d_1^t:E_1^{2-t,t}\rightarrow E_1^{3-t,t}$$ is bijective for $t\geq 2$ and injective for $t=1$.\\
Consider first the case when $t=1$. Since $d_1:E^{1,1}_1\rightarrow E^{2,1}_1$ is injective, then $\dim E^{1,1}\leq \dim E^{2,1}_1$. Moreover $\dim E_1^{1,1}=\dim H^2(K^*(f))_{2N}=\dim ER(f)_{2N-2}$ which is equal to $\tau(C)$ by Equation \ref{ER}. On the other hand, $\dim E^{2,1}_1(f)=\dim H^3(K^*(f))_{2N}=\dim M(f)_{2N-3}$. This proves the left hand side inequality in (A).\\
To prove the right hand side inequality, consider the limit term $$E_{\infty}^{2,1}=\frac{P^1H^2(U)}{P^2H^2(U)}.$$
It is known that  $P^sH^m(U) \supset F^sH^m(U)$, see \cite[Chapter 6]{Dbk}. For $s=1$ we have in addition $F^1H^2(U)=H^2(U)$, then  $P^1H^2= F^1H^2(U)=H^2(U)$. For $s=2$, $P^2H^2\supset F^2H^2(U)\supset H^2(U)$. It follows that the map $$Gr^1_F(H^2(U))\longrightarrow E^{2,1}_{\infty}$$ is an epimorphism, and hence, $\dim E^{2,1}_{\infty}\leq \dim Gr^1_F(H^2(U))$. 

By Theorem \ref{result1} $\dim Gr^1_F(H^2(U))=\sum_{j=1}^r g_j$. On the other hand, by Proposition 2.4 (iii) in \cite{DSt} the spectral sequence degenerates at the $E_2$ terms, i.e. $E_{\infty}^{2,1}=E_2^{2,1}=coker d_1^1$. Hence, $\dim E_{\infty}^{2,1}=\dim M(f)_{2N-3}-\tau(C)$, and this proves the inequality in $(A)$.\\
\textit{(B)} To prove the second inequality, we consider the differential $d_1^0:E_1^{2,0}\rightarrow E_1^{3,0}.$ We know that $E_{\infty}^{2,0}=E^{2,0}_2=\ker\ d_1^0 \simeq H^1(U)$ and hence $\dim (\ker\ d_1^0)=r-1$. In particular, $r-1 \leq E_1^{2,0}=\dim H^2(K^*(f))_N=\dim ER(f)_{N-2}$, and $\dim E_{\infty}^{3,0}=\dim (coker\ d_1^0)=g-\dim Er(f)_{N-2}+r-1$, where $g=\frac{(N-1)(N-2)}{2}=\dim M(f)_{N-3}.$ We compute now $b_2(U)$ in two different ways:
$$b_2(U)=\dim Gr^1_FH^2(U)+\dim Gr^2_FH^2(U)=\sum_{j=1}^r g_j+g-t,$$
by Theorem \ref{result1}, and 
$$b_2(U)=\dim E^{3,0}_{\infty}+\dim E^{2,1}_{\infty}=g-\dim ER(f)_{N-2}+r-1+\dim M(f)_{N-3}-\tau.$$
On the other hand, Theorem 1 in \cite{Dimca4} implies
$$\dim Er(f)_{N-2}=\dim M(f)_{2N-3}-\dim M(f_s)_{N-3}=\dim M(f)_{2N-3}-g.$$
The last two formulas imply that $b_2(U)=2g-\tau+r-1$. The first formula for $b_2(U)$ now implies that $\sum_{j=1}^rg_j-t=2g-\tau+r-1$. If we apply part $(A)$ of the theorem we get
$$\tau-g\leq \dim ER(f)_{N-2}\leq \tau+\sum g_j-g,$$
and this proves part $(B)$.
\end{proof}

\begin{example} (i) Let $C$ be the degree 5 curve defined by $f=xy(x+y)z^2+x^5+2y^5=0$. In this example, $C$ is an irreducible curve with exactly one triple point, and hence  $g_1=3$. A Singular computation gives $\dim M(f)_{7}=6$. This gives a strict inequality in Theorem \ref{result2} part (A), i.e. $$\dim M(f)_{7}-\tau(C)=2<3= g_1.$$
Moreover, the inequalities of part (B) of the theorem are
$$0\leq 0 \leq 1.$$

(ii) Let $C$ be the degree 9 curve defined by $f=(x^3+y^3+z^3)^3+(x^3+2y^3+3z^3)^3=0$. In this example, $C$ is a union of 3 smooth curves with $g_i=1$ for $i=1,2,3$. It has 9 ordinary triple points, and $\dim M(f)_{16}=\tau(C)=36$. This gives a strict inequality in Theorem \ref{result2} part (A), i.e. $$\dim M(f)_{16}-\tau(C)=0<3=\sum_{i=1}^3 g_i.$$
Moreover, the inequalities of part (B) of the theorem are
$$8\leq 8 \leq 9+2=11.$$
\end{example}

\begin{remark}
Part $(A)$ of Theorem \ref{result2} can be regarded as a generalization of Corollary \ref{rational nodal curves}, and  of the equation \ref{eqt}, since for nodal curves we have $F^2H^2(U)=P^2H^2(U)$. The above example shows that this equality may fail for curves with ordinary double and triple points. The previously known examples of curves with $F^2H^2(U)\neq P^2H^2(U)$ involved curves with non-ordinary singularities, see Examples 3.2, 3.3 and 3.4 in \cite{DSt}.
\end{remark}

\begin{remark}
Part $(B)$ of Theorem \ref{result2} is a generalization of Theorem 4.1 in \cite{DSt}. In this  result, since for rational curves $\dim ER(f)_{N-2}=r-1+t$, then each irreducible component $C_j$ and each triple point $P\in C$ yield one relation, and there is only one dependence relation among them. So one can ask about the possibility to write these syzygies in terms of the point $P$ and the defining functions of $C_j$ as in Theorem 4.1 \cite{DSt}.
\end{remark}
The following example shows that Corollary \ref{rational nodal curves} does not hold even for line arrangements with double and triple points.

\begin{example} \label{Pappus}

Consider the following two distinct realizations of the Pappus configuration $9_3$, see \cite{CS} and \cite{Dbk}, Example (6.4.16), p. 213. The first one is the line arrangement
$$\A_1: f=xyz(x-y)(y-z)(x-y-z)(2x+y+z)(2x+y-z)(-2x+5y-z)=0.$$
A Singular computation yields
$$HP(M(f))(t)=1+3t+6t^2+10t^3+15t^4+21t^5+28t^6+36t^7+42t^8+46t^9+48t^{10}+48t^{11}+$$
$$+47t^{12}+45(t^{13}+...$$
The second one is the line arrangement $\A_2$ given by
$$\A_2: g=xyz(x+y)(x+3z)(y+z)(x+2y+z)(x+2y+3z)(4x+6y+6z)=0.$$
Using again the Singular software, we get
$$HP(M(g))(t)=1+3t+6t^2+10t^3+15t^4+21t^5+28t^6+36t^7+42t^8+46t^9+48t^{10}+48t^{11}+$$
$$+46t^{12}+45(t^{13}+...$$
Both arrangements have $N=n=t=9$ and $HP(M(f))(t)-HP(M(g))(t)=t^{12} \ne 0.$
This shows once again that the curves with nodes and triple points are much more subtle than the nodal curves.
It also shows that it is rather difficult to control the dimension of the homogeneous components $M(f)_r$ for 
$r \ne 2N-3$.

\end{example}

\small \textbf{Acknowledegment:}I gratefully acknowledge the support of the Lebanese National Council for Scientific Research, without which the present study could not have been completed.


\begin{thebibliography}{99}


\bibitem{Bott & Tu} R. Bott, L. W. Tu: \textit{Differential Forms in Algebraic Topology}, (Springer-Verlag).

\bibitem{DCh}  A.D.R. Choudary, A. Dimca: \textit{Koszul Complexes and Hypersurface Singularities}, Proc. Amer. Math. Soc. 121(1994), 1009-1016. 

\bibitem{CS} D. Cohen, A. Suciu: \textit{On Milnor fibrations of arrangements}.
J. London Math. Soc. (2) 51 (1995), no. 1, 105–119.

\bibitem{Singular} W. Decker, G.-M. Greuel, G. Pfister, H. Sch¨onemann: Singular 3-1-3 — A computer
algebra system for polynomial computations.http://www.singular.uni-kl.de (2011).

\bibitem{Del}
P. Deligne: \textit{Th\'eorie de Hodge II}, Publ. Math. IHES, 40 (1971), 5--58.
\bibitem{Dbk}  A. Dimca: \textit{Singularities and Topology of Hypersurfaces} (Universitext, Springer-Verlag, 1992).

\bibitem{Dimca}  A. Dimca: \textit{Syzygies of Jacobian Ideals and Defects of Linear Systems}, Bull. Math. Soc. Sci. Math. Roumanie, 2012. Soc. 1-13, 2012

\bibitem{Dimca2} A. Dimca: \textit{On the Milnor Fibrations of Weighted Homogeneous Polynomials}, Composito Math. 76(1990), 19-47.

\bibitem{Dimca4}  A. Dimca: \textit{Monodromy of Triple Point Line Arrangements}, arXiv:1107.2214, 2012


\bibitem{DSt}  A. Dimca, G. Sticlaru: \textit{Koszul Complexes and Pole Order Filtrations}, arXiv:1108.3976, to appear in Proc. Edinburg Math. Soc.

\bibitem{DSt2}  A. Dimca, G. Sticlaru: \textit{Chebyshev Curves, Free Resolutions and Rational Curve Arrangements}, Math. Proc. Camb. Phil. Soc. 1-13, 2012

\bibitem{DS} A. Dimca, M. Saito: \textit{A Generalization on Griffiths' Theorem on Rational Integrals},Duke Math. J. 135(2006),303-326.

\bibitem{DSW} A. Dimca, M. Saito, L. Wotzlaw: \textit{A Generalization on Griffiths' Theorem on Rational Integrals II}, Michigan Math J.58(2009), 603-625.


\bibitem{Gr} Ph. Griffiths: \textit{On the Period of certain Rational Integrals I, II}, Ann. Math. 90(1969), 460-541.


\bibitem{Hartshorne}  R. Hartshorne: \textit{Algebraic Geometry} (GTM 52, Springer 1977).

\bibitem{Se} E. Sernesi: \textit{The Local Cohomology of the Jacobian Ring}, arXiv: 1306.3736v3.

\bibitem{Voisin} C. Voisin: \textit{Th\'eorie de Hodge et G\'eom\'etrie alg\'ebrique complexe}, Soci\'et\'e math\'ematique de France 2002.




\end{thebibliography}
\end{document}